\documentclass[12pt]{article} % for a short document

\usepackage[margin=3cm]{geometry}
\usepackage[english]{babel}
\usepackage{hyperref}
\usepackage{amsmath}
\usepackage{amsthm}
\usepackage{amssymb}
\usepackage{ucs}
\usepackage{enumerate}
\usepackage[utf8x]{inputenc}
\usepackage[T1]{fontenc}
\usepackage{color}
\usepackage{tikz}
\usepackage{diagbox}
\DeclareMathOperator{\rng}{\mathrm{rng}}
\DeclareMathOperator{\supp}{\mathrm{supp}}

\newcommand{\Aut}{\mathrm{Aut}}

  \newcommand{\R}{\mathbb R}
  \newcommand{\N}{\mathbb N}

  \newcommand{\LL}{\mathrm L}

 \newcommand{\dom}{\mathrm{dom}\;}

  \newcommand{\inv}{^{-1}}

  \renewcommand{\leq}{\leqslant}
  \renewcommand{\geq}{\geqslant}
  \newcommand{\abs}[1]{\left\lvert #1\right\rvert}

  \newcommand{\impl}{\Rightarrow}

  \newcommand{\la}{\left\langle}
  \newcommand{\ra}{\right\rangle}

  \DeclareMathOperator{\diam}{\mathrm{diam}}

\newtheorem{thmi}{Theorem}

\newtheorem{thm}{Theorem}[section]
\newtheorem{cor}[thm]{Corollary}
\newtheorem{lem}[thm]{Lemma}
\newtheorem{prop}[thm]{Proposition}

\theoremstyle{definition}

\newtheorem{qu}[thm]{Question}

\newtheorem{df}[thm]{Definition}
\newtheorem*{rmq}{Remark}

\newtheorem{exemple}[thm]{Example}

\title{Bounded normal generation is not equivalent to topological bounded normal generation}

\author{Philip~A. Dowerk\footnote{Research supported by ERC CoG No. 614195 and by  ERC CoG No. 681207.}~ and François Le Maître\footnote{Research supported by Projet ANR-14-CE25-0004 GAMME.}}

\begin{document}

\maketitle

\begin{abstract}
We show that some derived $\LL^1$ full groups provide examples of non simple Polish groups with the topological bounded normal generation property. In particular, it follows that there are Polish groups with the topological bounded normal generation property but not the bounded normal generation property. 
\end{abstract}

%\tableofcontents

\section{Introduction}

The study of how quickly the conjugacy class of a non-central element generates the whole group has often turned out to be useful in understanding both the algebraic and topological structure of the given group. For example, such results on involutions by Broise \cite{MR0223900} were crucial in de la Harpe's proof of simplicity of projective unitary groups of some von Neumann factors \cite{MR535064}, while analogous results by Ryzhikov \cite{MR823423} are used in the proof automatic continuity of homomorphisms  for full groups of measure-preserving equivalence relations by Kittrell and Tsankov \cite{MR2599891}.

A group $G$ has \emph{bounded normal generation} (BNG) if the reunion of the conjugacy classes of any nontrivial element  and of its inverse generate $G$ in finitely many steps. Observe that any group with bounded normal generation is simple. Bounded normal generation already appears in Fathi's result that the group of measure-preserving transformations of a standard probability space is simple \cite{MR0486410}, but the name was only given recently in \cite{Dow15,DowThom15}.

Obtaining bounded normal generation with  a sharp number of steps is a crucial point of the seminal article \cite{MR1865975} by Liebeck and Shalev concerning bounds on the diameter of the Cayley graph of a finite simple group. Other applications of bounded normal generation can be found in the first named author's work together with Thom \cite{Dow15,DowThom15} concerning invariant automatic continuity and the uniqueness of Polish group topologies for projective unitary groups of certain operator algebras, in particular for type II$_1$ von Neumann factors.

The notion of bounded normal generation can be topologized: a topological group $G$ has \emph{topological bounded normal generation} if the conjugacy classes of every nontrivial element and of its inverse generate a dense subset of $G$ in finitely many steps. 
This definition was introduced in \cite{Dow15} to give a clean structure to the proof of bounded normal generation for projective unitary groups of II$_1$ factors: prove topological bounded normal generation via some finite-dimensional approximation methods in the strong operator topology, then use some commutator tricks and results on symmetries to obtain bounded normal generation. 

However, it is a priori not clear whether bounded normal generation and topological bounded normal generation are the same properties. Basic examples of non simple topologically simple groups such as the symmetric group over an infinite set or the projective unitary group of an infinite dimensional Hilbert space actually fail topological bounded normal generation (see Cor. \ref{cor: unitary und symmetric fail TBNG}).
In fact, this problem was left open in \cite{Dow15} and is the main point of this  article: we provide examples that separate bounded normal generation from its topological counterpart. Actually, we show that there are Polish groups which have topological bounded normal generation but which are not even simple.

The groups that serve our purposes are  called derived L$^1$ full groups, introduced by the second named author in \cite{l1fullgpsI} as a measurable analogue the derived groups of small topological full groups. Let us now state our main result, see Section \ref{sec_prelim} for definitions. 

\begin{thmi}[{see Thm.  \ref{thm: TBNG for amenable}}]\label{thm: hyperfinite has TBNG}
 The derived L$^1$ full group of any hyperfinite ergodic graphing has topological bounded normal generation. 
\end{thmi} 

In the above theorem we can moreover estimate how many products of conjugacy classes of a non-trivial element and its inverse in the derived L$^1$ full group we need to generate a dense subset, see Thm. \ref{thm: TBNG for amenable}.  

We do not know wether one can remove the hyperfiniteness asumption in Theorem \ref{thm: hyperfinite has TBNG}. As a concrete example of this result, one can consider the derived $\LL^1$ full group of a measure-preserving ergodic transformation, which was already known to be non-simple but topologically simple \cite{l1fullgpsI}. Using reconstruction-type results, we can show that this provides many examples of non simple groups with topological bounded normal generation.

\begin{thmi}[{see Cor. \ref{cor: continuum counterexamples}}]\label{thm: uncountably many cex}
There are uncountably many pairwise non-isomorphic Polish groups which have topological bounded normal generation but fail to be simple (in particular they fail bounded normal generation).
\end{thmi} 

The article is structured as follows: in Section \ref{sec_prelim} we provide preliminaries around bounded normal generation and define derived L$^1$ full groups of graphings, in Section \ref{sec_invo} we show that for hyperfinite graphings these groups contain a dense subgroup in which every element is a product of two involutions, in Section \ref{sec_TBNG} we use this to prove Theorem \ref{thm: hyperfinite has TBNG} (see Thm. \ref{thm: TBNG for amenable}) and finally in Section \ref{sec_recon} we prove a reconstruction theorem for derived L$^1$ full groups and use it to prove Theorem \ref{thm: uncountably many cex} (see Cor. \ref{cor: continuum counterexamples}).

\section{Preliminaries}\label{sec_prelim}

\subsection{Topological bounded normal generation}\label{sec:define TBNG}

If $G$ is a group and $g\in G$, we write $g^{\pm G}=\{hgh\inv: h\in G\}\cup\{hg\inv h\inv: h\in G\}$. Given a subset $A$ of $G$, we moreover write $A^{\cdot n}$ for the set of products of $n$ elements of $A$, that is 
$$A^{\cdot n}=\{a_1\cdots a_n: a_1,...,a_n\in A\}.$$

The following definitions are taken from \cite{Dow15, DowThom15}.

\begin{df}
A  group $G$ has \textbf{bounded normal generation} (BNG) if for every non-trivial $g\in G$, there is $n(g)\in\N$ such that $$\left(g^{\pm G}\right)^{\cdot n(g)}=G.$$
A function $n:G\setminus\{1\}\rightarrow \mathbb{N}$ satisfying the above assumption  is then called a \textbf{normal generation function}. 
\end{df}

Observe that BNG implies simplicity. An example of a simple group without BNG is given by finitely supported permutations of signature $1$ on an infinite set. 
Let us now define the topological version of BNG.

\begin{df}
A topological group $G$ has \textbf{topological bounded normal generation} (TBNG) if for every non-trivial $g\in G$, there is $n(g)\in\N$ such that $$\overline{(g^{\pm G})^{\cdot n(g)}}=G.$$
We still call a function $n:G\setminus\{1\}\rightarrow\mathbb{N}$ satisfying the above assumption a normal generation function.
\end{df}

The purpose of this paper is to exhibit examples of non-simple topological groups with TBNG.
Let us first explain why the first examples of non simple topologically simple groups which come to mind actually fail TBNG.

\begin{df}
Let $G$ be a topological group. A \textbf{lower semi-continuous invariant length function} on $G$ is a map 
$l:G\to[0,+\infty]$ such that
\begin{enumerate}[(i)]
\item $l(1_G)=0$;
\item $l(g\inv)=l(g)$ for all $g\in G$;
\item $l(gh)\leq l(g)+l(h)$ for all $g,h\in G$;
\item $l(ghg\inv)=l(h)$  for all $g,h\in G$;
\item for all $r\geq 0$, the set $l\inv([0,r])$ is closed.
\end{enumerate}
Such a function is \textbf{unbounded} if for all $K\in\N$ there is $g\in G$ such that $l(g)\in]K,+\infty[$.
\end{df}

\begin{exemple}
On the symmetric group over a discrete set $X$ endowed with the pointwise convergence topology, one can let $l(g)$ be the cardinality of the support of $g$. If $X$ is infinite, then such a length function is unbounded.

On the  unitary group of a Hilbert space endowed with the strong topology, one can let $l(g)$ be the minimum over $\lambda\in\mathbb S^1$ of the rank of $(g-\lambda \mathrm{id}_{\mathcal H})$ or of the trace class norm of $(g-\lambda\mathrm{id}_{\mathcal H})$. If the Hilbert space is infinite dimensional, then such a length function is unbounded and quotients down to an unbounded lower semi-continuous length function on the projective unitary group. 
\end{exemple}

Observe that if we have an invariant length function on $G$ then the group of elements of finite length is a $F_\sigma$ normal subgroup of $G$. Moreover, we have the following straightforward proposition.

\begin{prop}
If a topological group $G$ admits an unbounded lower semi-continuous invariant length function, then it fails TBNG. 
\end{prop}
\begin{proof}
Let $l$ be an unbounded lower semi-continuous invariant length function on $G$.
By unboundness we find $g\in G$ with $l(g)\in ]0,+\infty[$ (in particular $g$ is non-trivial). Fix $n\in\N$ and let $N=n\cdot l(g)$. Observe that $(g^{\pm G})^{\cdot n}\subseteq l\inv([0,N])$, and since the latter is closed we also have $\overline{(g^{\pm G})^{\cdot n}}\subseteq l\inv([0,N])$. Since $l$ is unbounded, $\displaystyle l\inv([0,N])$ is not equal to $G$ and we conclude that for every $n\in\N$ we have $\displaystyle\overline{(g^{\pm G})^{\cdot n}}\neq G$. So $G$ fails TBNG.
\end{proof}%

Applying this proposition to our two examples, we get the following.
\begin{cor}\label{cor: unitary und symmetric fail TBNG}
The symmetric group over an infinite discrete set fails TBNG for the topology of pointwise convergence. The projective unitary group of an infinite-dimensional Hilbert space fails TBNG for the strong topology.
\end{cor}

\begin{rmq}
 The fact that the projective unitary group of an infinite-dimensional Hilbert space fails TBNG can also be seen via generalized projective $s$-numbers as in \cite[p. 80-81]{Dow15}.
\end{rmq}

Let us conclude this section by introducing another natural topological version of BNG. Define \textbf{strong topological bounded normal generation} (STBNG) for a topological group $G$ by asking that for every non-trivial $g\in G$, there is $n(g)\in\N$ such that $$\left(\overline{g^{\pm G}}\right)^{\cdot n(g)}=G.$$
Since we always have $\left(\overline{g^{\pm G}}\right)^{\cdot n}\subseteq \overline{(g^{\pm G})^{\cdot n}}$, STBNG implies TBNG.
In a compact group finite pointwise products of closed subsets are compact hence closed, so we actually always have $\left(\overline{g^{\pm G}}\right)^{\cdot n}= \overline{(g^{\pm G})^{\cdot n}}$. Hence STBNG is equivalent to TBNG for compact groups. 
However it is not true in general that products of closed sets are closed (for instance in $G=\R$ consider $A=B=\N\cup-\sqrt2 \N$), so the following question is natural.

\begin{qu}
Is it true in general that TBNG is equivalent to STBNG ?
\end{qu}

\subsection{Derived \texorpdfstring{$\LL^1$}{L1} full groups}

Let $(X,\mu)$ be a standard probability space. We will ignore null sets, in particular we identify two Borel functions on $(X,\mu)$ if they are the same up to a null set.
Given $A,B$ Borel subsets of $X$, a \textbf{partial isomorphism} of $(X,\mu)$ of \textbf{domain} $A$ and \textbf{range} $B$ is a Borel bijection $f: A\to B$ which is measure preserving for the measures induced by $\mu$ on $A$ and $B$ respectively. We denote by $\dom f=A$ its domain, and by $\rng f=B$ its range. Note that in particular, $\mu(\dom f)=\mu(\rng f)$. When $\dom f=\rng f=X$ we say that $f$ is a \textbf{measure-preserving transformation} and we denote by $\Aut(X,\mu)$ the group of such transformations.

A \textbf{graphing} is a countable set of partial isomorphisms of $(X,\mu)$, often denoted $\Phi=\{\varphi_1,...,\varphi_k,...\}$ where the $\varphi_k$'s are partial isomorphisms. It \textbf{generates} a \textbf{measure preserving equivalence relation} $\mathcal R_\Phi$, defined to be the smallest equivalence relation such that $\varphi(x)\; \mathcal R_\Phi \; x$ for every $\varphi\in \Phi$.

Every graphing $\Phi$ induces a natural graph structure on $X$ where we connect $x$ to $y$ whenever there is $\varphi\in\Phi$ such that $y=\varphi(x)$. The connected components of this graph are precisely the $\mathcal R_\Phi$ classes, and we thus have a \textbf{graph metric} $d_\Phi: \mathcal R_\Phi\to\N$ given by $$d_\Phi(x,y)=\min\{n\in \N: \exists \varphi_1,...,\varphi_n\in\Phi, \epsilon_1,...,\epsilon_n\in\{-1,+1\}, y=\varphi_n^{\epsilon_n}\circ\cdots\circ\varphi_1^{\epsilon_1}(x)\}.$$

The \textbf{full group} of the measure-preserving equivalence relation $\mathcal R_\Phi$ is the group  of all measure-preserving transformations $T$ such that for all $x\in X$ we have $(x,T(x))\in\mathcal R_\Phi$. We denote it by $[\mathcal R_\Phi]$.

The $\LL^1$ full group of the graphing $\Phi$ is the group of all $T\in[\mathcal R_\Phi]$ such that 
$$\int_Xd_\Phi(x,T(x))d\mu(x)<+\infty.$$
It is denoted by $[\Phi]_1$. It has a natural complete separable right-invariant metric $\tilde d_\Phi$ given by 
$$\tilde d_\Phi(T,U)=\int_Xd_\Phi(T(x),U(x))d\mu(x).$$
We refer the reader to \cite{l1fullgpsI} for proofs and more background.

\begin{df}The \textbf{derived} $\LL^1$ \textbf{full group} of a graphing $\Phi$ is the closure of the derived group of the $\LL^1$ full group of $\Phi$. It is denoted by $[\Phi]_1'$. 
\end{df}

A graphing is \textbf{aperiodic} if all its connected components are infinite.
Recall that all the involutions from $[\Phi]_1$ actually belong to $[\Phi]_1'$ as soon as $\Phi$ is aperiodic (see \cite[Lem. 3.10]{l1fullgpsI}). 
The following result will be strengthened in the next section for hyperfinite graphings.

\begin{thm}[{\cite[Thm. 3.13]{l1fullgpsI}}]\label{thm: topo gen by I_phi}
Let $\Phi$ be an aperiodic graphing. Then its derived $\LL^1$ full group $[\Phi]_1'$ is topologically generated by involutions. 
\end{thm}

\section{Products of two involutions in the derived \texorpdfstring{$\LL^1$}{L1}  full group of hyperfinite graphings}\label{sec_invo}

\begin{df}
A measure-preserving equivalence relation is called \textbf{finite} if all its equivalence classes are finite. 

It is \textbf{hyperfinite} if it can be written as an increasing union of measure-preserving equivalence relations $\mathcal R_n$ whose equivalence classes are finite. 

A graphing $\Phi$ is hyperfinite if the equivalence relation it generates is hyperfinite.
\end{df}

Let $\Phi$ be a hyperfinite graphing, our aim is to show that its derived $\LL^1$ full group contains a dense subgroup all whose elements are products of two involutions. Our proof uses the exact same ideas as the proof that $\LL^1$ full groups of hyperfinite graphings are amenable \cite{L1fullgpsIIinprep}; we provide full details for the reader's convenience.

\begin{lem}\label{lem: dense increasing union}Let $\Phi$ be an aperiodic graphing and suppose that $\mathcal R_\Phi$ is written as an increasing union of equivalence relations $\mathcal R_n$. Then
$\bigcup_{n\in\N}[\mathcal R_n]\cap [\Phi]_1'$
is dense in $[\Phi]_1'$. 
\end{lem}
\begin{proof}
By  we only need to be able to approximate involutions from  $[\Phi]_1'$ by elements from $\bigcup_{n\in\N}[\mathcal R_n]\cap [\Phi]_1'$. But this is immediate from the dominated convergence theorem and the fact that the $U$-invariant sets $A_n=\{x\in X: x\mathcal R_n U(x)\}$ satisfy $\bigcup_{n\in\N} A_n= X$ since $\bigcup_n \mathcal R_n=\mathcal R_\Phi$. 
\end{proof}

\begin{df}
Given a graphing $\Phi$, we say that the $\mathcal R_\Phi$-classes have a \textbf{uniformly bounded $\Phi$-diameter} if there is $N\in\N$ such that for any $(x,y)\in\mathcal R_\Phi$ we have $d_\Phi(x,y)\leq N$.
\end{df}

\begin{lem}\label{lem: bounded classes}Let $\Phi$ be an aperiodic graphing, let $\mathcal R$ be a finite subequivalence relation of $\mathcal R_\Phi$ whose equivalence classes have a uniformly bounded $\Phi$-diameter. Then $[\mathcal R]$ is a subgroup of $[\Phi]_1'$.
\end{lem}
\begin{proof}
The fact that there is a uniform bound $M$ on the diameter of the $\mathcal R$-classes ensures us that $[\mathcal R]$ is a subgroup of $[\Phi]_1$. Moreover all the elements of $[\mathcal R]$ are periodic and hence belong to $[\Phi]_1'$.
\end{proof}

\begin{lem}\label{lem: uniform bounded on cardinality implies loc finite}
Let $\mathcal R$ be an equivalence relation. Suppose that there is $n\in\N$ such that all the $\mathcal R$-classes have cardinality at most $n$. Then $[\mathcal R]$ is locally finite. 
\end{lem}
\begin{proof}
Let $U_1,...,U_k\in[\mathcal R]$ for some arbitrary $k\in\mathbb{N}$. By using a Borel linear order on $X$, we can identify in an $\mathcal R$-invariant manner the $\mathcal R$ orbit of each $x\in X$ with a set of the form $\{1,...,m_x\}$ with $m_x\leq n$. For each $x\in X$ there are only finitely many ways that the marked group generated by $\la U_1,...,U_k\ra$ can act on the finite set $\{1,...,m_x\}$. So we get a finite partition of $X$ into $\mathcal R$-invariant sets such that on each atom of this partition, the marked group $\la U_1,...,U_k\ra$ acts on each orbit the same way, and we thus embed $\la U_1,...,U_k\ra$ in a finite product of finite permutation groups. 
\end{proof}

\begin{thm}\label{thm: dense products of 2 invol}
Let $\Phi$ be a hyperfinite graphing. Then $[\Phi]_1'$ contains a dense locally finite subgroup, all whose elements are products of two involutions.
\end{thm}
\begin{proof}
By definition we may write $\mathcal R_\Phi$ as an increasing union of finite equivalence relations $\mathcal R_n$. It is well known that we can moreover assume that each $\mathcal R_n$ has equivalence classes of size at most $n$. Let us show that we can moreover assume each $\mathcal R_n$ has equivalence classes of uniformly bounded $\Phi$-diameter and cardinality.

We first find an increasing sequence of integers $(\varphi(n))_{n\in\N}$ such that for all $n\in\N$, $$\mu(\{x\in X: \abs{[x]_{\mathcal R_n}}>\varphi(n)\text{ or }\diam_\Phi([x]_{\mathcal R_n})>\varphi(n)\})<\frac 1{2^n}.$$
Then by the Borel-Cantelli lemma, for almost all $x\in X$ there are only finitely many $n\in\N$ such that $\abs{[x]_{\mathcal R_n}}>\varphi(n)$ or $\diam_\Phi([x]_{\mathcal R_n})>\varphi(n)$. So if we define  new equivalence relations $\mathcal S_n$ by $(x,y)\in\mathcal S_n$ if $(x,y)\in\mathcal R_n$ and  $\diam_\Phi([x]_{\mathcal R_n})\leq\varphi(n)$ and $\abs{[x]_{\mathcal R_n}}\leq\varphi(n)$, we still have $\bigcup_{n\in\N} \mathcal S_n=\mathcal R$ and the $\mathcal S_n$-classes have a uniformly bounded diameter and cardinality as wanted.

By Lemma \ref{lem: dense increasing union}, $\bigcup_{n\in\N}[\mathcal S_n]\cap [\Phi]_1'$ is dense in $[\Phi]_1'$. But Lemma \ref{lem: bounded classes} yields that $[\mathcal S_n]\cap [\Phi]_1'=[\mathcal S_n]$. Since $\mathcal S_n$ is finite and all its classes have cardinality at most $\varphi(n)$, $[\mathcal S_n]$ is locally finite by Lem. \ref{lem: uniform bounded on cardinality implies loc finite}. Finally each element of $[\mathcal S_n]$ is a product of two involutions by \cite[Sublemma 4.3]{MR2583950}.
So the subgroup $\bigcup_{n\in\N}[\mathcal S_n]$ is as wanted. 
\end{proof}

\section{Topological bounded normal generation for derived \texorpdfstring{$\LL^1$}{L1}  full groups}\label{sec_TBNG}

In this section, we will prove that derived $\LL^1$ full groups of ergodic amenable graphings have TBNG. A key tool will be Thm. \ref{thm: dense products of 2 invol}, but we also need a better understanding of products of conjugates of involutions. 

\begin{lem}[{\cite[Lem. 3.21]{l1fullgpsI}}]
Let $\Phi$ be an ergodic graphing, and let $U\in[\Phi]_1$ be an involution whose support has measure $\alpha\leq 1/2$. Then the closure of the $[\Phi]_1'$ conjugacy class of $U$ contains  all involutions in $[\Phi]_1$ whose support has measure $\alpha$ and is disjoint from the support of $U$. 
\end{lem}

\begin{lem}\label{lem:closure conjugacy class contains involutions}
Let $\Phi$ be an ergodic graphing, let $U\in [\Phi]_1$ be an involution whose support has measure $\alpha<1/3$. Then the closure of the $[\Phi]_1'$ conjugacy class of $U$ contains  all involutions in $[\Phi]_1$ whose support has measure $\alpha$.
\end{lem}
\begin{proof}
Let $V\in[\Phi]_1$ have a support of measure $\alpha$. Observe that $\mu(\supp U\cup \supp V)<2/3$. 
By ergodicity, we find an involution $W\in[\mathcal R_\Phi]$ whose support has measure 1/3 and is disjoint from $\supp U\cup \supp V$. By shrinking down $W$, we may actually assume that $W\in[\Phi]_1$ and that $\mu(\supp W)=\alpha$.

By the previous lemma $W$ belongs to the closure of the conjugacy class of $U$ and $V$ belongs to the closure of the conjugacy class of $W$ so $V$ belongs to the closure of the conjugacy class of $U$. 
\end{proof}

\begin{lem}\label{lem: smaller support}
Let $\Phi$ be an ergodic graphing and let $U\in [\Phi]_1$ be an involution whose support has measure $\alpha<1/3$. Then every involution whose support has measure $\leq 2\alpha$ is the product of $2$ elements from the closure of the $[\Phi]_1'$ conjugacy class of $U$.
\end{lem}
\begin{proof}
Let $V\in[\Phi]_1$ be an involution whose support has measure $\beta \leq 2\alpha$. 
Cut the support of $V$ into two disjoint $V$-invariant sets $A_1$ and $A_2$ of measure $\beta/2$, and let $W\in[\Phi]_1$ be an involution with support of measure $\alpha-\beta/2$ disjoint from the support of $V$ (in particular, $W$ commutes with $V$).

Then $V=V_{A_1}V_{A_2}=(V_{A_1}W)(V_{A_2}W)$ and by the previous lemma both $V_{A_1}W$ and $V_{A_2}W$ belong to the closure of the conjugacy class of $U$. 
\end{proof}

\begin{thm}\label{thm: TBNG for amenable}
Let $\Phi$ be a hyperfinite ergodic graphing. Then $[\Phi]_1'$ has TBNG. A normal generation function is given by $n(T)= 2+2\lceil \frac{7}{2\mu(\supp T)}\rceil$.
\end{thm}
\begin{proof}
Let $T\in[\Phi]_1'$. Let $A\subseteq \supp T$ be a maximal subset such that $A$ and $T(A)$ are disjoint. Then by maximality $\supp T\subseteq T\inv(A)\cup A\cup T(A)$, and thus  $\mu(A)\geq \mu(\supp T)/3$. 

We may now find an involution $U\in[\Phi]_1$ whose support is contained in $A$ and has measure $\mu(\supp T)/7$. 
Now $[T,U]$ is an involution whose support has measure $\mu(\supp [T,U])=2\mu(\supp T)/7<1/3$. 

Let $V$ be an arbitrary involution, we cut down its support into $V$-invariant pieces of measure $\mu(\supp [T,U])$ plus a remaining piece of measure $<\mu(\supp [T,U])$. Observe that $V$ is equal to the product of the transformations it induces on these pieces. There are at most $\lceil\frac 1{\mu(\supp [T,U])}\rceil=\lceil\frac 7{2\mu(\supp T)}\rceil$ such pieces, and by Lem. \ref{lem:closure conjugacy class contains involutions} and Lem. \ref{lem: smaller support} applied to the involutions induced by $V$ on the pieces, we can then write $V$ as a product of at most $1+\lceil\frac 7{2\mu(\supp T)}\rceil$  elements of the closure of the conjugacy class of $[T,U]$. 

Since there is a dense subgroup of $[\Phi]_1'$ whose elements are products of $2$ involutions, we have a dense subgroup of $[\Phi]_1'$ all whose elements are in $\overline{(T^{\pm[\Phi]_1'})^k}$ where
$$k=2\left(1+\lceil\frac 7{2\mu(\supp T)}\rceil\right)$$
So  $[\Phi]_1'$ has TBNG.
\end{proof}

\begin{rmq}
To get STBNG via the same approach, we would need to know that each element of $[\Phi]_1'$ is a product of at most $k$ involutions for some fixed $k$. We do not know wether this is true, and we also cannot exclude that every element of the derived $\LL^1$ full group is a product of two involutions.
\end{rmq}

\begin{cor}\label{cor: derived L1 of T has TBNG but not simple}
Let $T$ be an ergodic measure-preserving transformation. Then $[T]_1'$ has TBNG but is not simple. 
\end{cor}
\begin{proof}
By \cite[Thm. 1]{MR0131516} the graphing $\{T\}$ is hyperfinite, so by Theorem \ref{thm: TBNG for amenable} the derived $\LL^1$ full group $[T_\alpha]_1'$ has TBNG. But by \cite[Thm. 4.26]{l1fullgpsI} we also have that $[T]'_1$ is not simple.
\end{proof}

\section{Reconstruction for the derived \texorpdfstring{$\LL^1$}{L1}  full group}\label{sec_recon}

Recall that two measure-preserving transformations $T$ and $T'$ are \textbf{conjugate} if there is a measure-preserving transformation $S$ such that $T=ST'S\inv$ or $T\inv=S T'S\inv$. In \cite{l1fullgpsI} it was shown that when the $\LL^1$ full groups of two ergodic measure-preserving transformations $T$ and $T'$ are isomorphic, then $T$ and $T'$ are flip-conjugate.
The purpose of this section is to prove the same result for the \textit{derived} $\LL^1$ full group. 

Let us first note that by \cite[Prop. 3.16]{l1fullgpsI} and a reconstruction theorem of Fremlin as stated in \cite[Thm. 3.17]{l1fullgpsI}, we have the following result:

\begin{lem}
Let $\Phi$ and $\Psi$ be a aperiodic graphings, and let $\rho:[\Phi]_1'\to[\Psi]_1'$ be a group isomorphism. Then there is a non-singular transformation $S$ such that for all $U\in[\Phi]_1'$ we have 
$\rho(U)=SUS\inv$.\qed
\end{lem}

In the ergodic case, we can upgrade this as in \cite[Cor. 3.18]{l1fullgpsI}.

\begin{cor}
Let $\Phi$ and $\Psi$ be a ergodic graphings, and let $\rho:[\Phi]_1'\to[\Psi]_1'$ be a group isomorphism. Then there is a measure-preserving transformation $S$ such that for all $U\in[\Phi]_1'$ we have 
$\rho(U)=SUS\inv$. 
\end{cor}
\begin{proof}
Same as the proof of \cite[Cor. 3.18]{l1fullgpsI}.
\end{proof}

We then have the following reconstruction theorem, whose proof is inspired from a similar result of Bezuglyi and Medynets for topological full groups \cite[Lem. 5.12]{MR2353913}.

\begin{thm}
Let $T_a$ and $T_b$ be two ergodic measure-preserving transformations. Then the following are equivalent:
\begin{enumerate}[(i)]
\item $T_a$ and $T_b$ are flip-conjugate;
\item $[T_a]_1'$ and $[T_b]_1'$ are topologically isomorphic.
\item $[T_a]_1'$ and $[T_b]_1'$ are abstractly isomorphic;
\end{enumerate}
\end{thm}
\begin{proof}
The only non-trivial implication is (iii)$\impl$(i). Let $\rho:[T_a]_1'\to[T_b]_1'$ be a group isomorphism. By the previous corollary we find a measure-preserving transformation $S$ such that for all $U\in [T_a]_1'$ one has 
$$\rho(U)=SUS\inv.$$
We then find $A,B\subseteq X$ such that $X=A\sqcup T_a(A)\sqcup B\sqcup T_a(B)\sqcup T_a^2(B)$ (see e.g. \cite[Prop. 2.7]{l1fullgpsI}) and let
$$A_1=A, A_2=T_a(A), A_3=B, A_4=T_a(B), A_5=T_a^2(B).$$
Consider then the involutions $U_1, \cdots, U_5$ defined by $$U_i(x)=\left\{\begin{array}{cc}T_a(x) & \text{if }x\in A_i \\T_a\inv(x) & \text{if }x\in T_a(A_i) \\x & \text{else.}\end{array}\right.$$
These involutions all belong to the $\LL^1$ full group of $T_a$, so they actually belong to $[T_a]_1'$ by \cite[Lem. 3.10]{l1fullgpsI}. Now for every $x\in A_i$ we have $T_a(x)=U_i(x)$ so for every $x\in S(A_i)$ we have $ST_aS\inv(x)=SU_iS\inv(x)$. Since for each $i\in\{1,...,5\}$ we have $SU_iS\inv\in[T_b]'_1\leq [T_b]_1$, we conclude that $ST_aS\inv$ can be obtained by cutting and pasting finitely many elements of $[T_b]_1$, hence $ST_aS\inv\in[T_b]_1$.

By the same argument  we also have $T_b\in [ST_aS\inv]_1$ so $T_b$ and $ST_aS\inv$ share the same orbits.  Belinskaya's theorem (as stated in \cite[Thm. 4.1]{l1fullgpsI}) now implies that $T_a$ and $T_b$ are flip-conjugate. 
\end{proof}

\begin{cor}\label{cor: continuum counterexamples}
There is a continuum of pairwise non-isomorphic non-simple Polish groups having TBNG. In particular, TBNG does not imply BNG.
\end{cor}
\begin{proof}
Consider, for an irrational $\alpha\in]0,1/2[$, the rotation $T_\alpha(x)=x \mod 1$. By Cor. \ref{cor: derived L1 of T has TBNG but not simple} the derived $\LL^1$ full groups $[T_\alpha]_1'$ have TBNG and are not simple. Moreover all those groups are non-isomorphic  by the previous theorem and the fact that two irrational rotations $T_\alpha$ and $T_\beta$ are conjugate iff $\alpha=\pm\beta \mod 1$. 
\end{proof}

%%%%%---------------------------------------------------------------------------------------------------------%%%%%

\bibliographystyle{alpha}
\bibliography{/Users/francoislemaitre/Dropbox/Maths/biblio}

\end{document}